\documentclass{article}
\usepackage{graphicx} 
\usepackage{hyperref}
\usepackage{amsmath}
\usepackage{amssymb}
\usepackage{amstext}
\usepackage{amsfonts}
\usepackage{amsthm}
\usepackage[numbers]{natbib}
\usepackage{siunitx}
\usepackage[algo2e]{algorithm2e}
\usepackage{algorithmic}
\usepackage{algorithm, caption}

\usepackage{pythonhighlight}
\usepackage{matlab-prettifier}
\usepackage{verbatimbox}
\usepackage{fancyvrb}

\newcommand{\complex}{{\mathbb C}}
\newcommand{\C}{{\mathbb C}}
\newcommand{\R}{{\mathbb R}}

\newtheorem{theorem}{Theorem}
\newtheorem{corollary}[theorem]{Corollary}

\newtheorem*{remark*}{Remark}

\newtheorem*{example*}{Example}

\DeclareMathOperator{\offdiag}{offdiag}
\DeclareMathOperator{\diag}{diag}

\newcommand{\randdiag}{RandDiag}

\setcitestyle{number,open={[},close={]}}
\title{A simple, randomized algorithm \\ for diagonalizing  normal matrices\footnote{This work was supported by the SNSF research project \emph{Probabilistic methods for joint and
singular eigenvalue problems}, grant number: 200021L\_192049. }}
\date{\today}

\author{Haoze He\footnotemark[2] \and Daniel Kressner\footnote{\'Ecole Polytechnique F\'ed\'erale de Lausanne (EPFL), Institute of Mathematics, 1015 Lausanne, Switzerland. E-mails: \href{mailto:haoze.he@epfl.ch}{haoze.he@epfl.ch},  \href{mailto:daniel.kressner@epfl.ch}{daniel.kressner@epfl.ch}}}
\begin{document}

\maketitle

\begin{abstract}
 We present and analyze a simple numerical method that diagonalizes a complex normal matrix $A$ by diagonalizing the Hermitian matrix obtained from a random linear combination of the Hermitian and skew-Hermitian parts of $A$.
\end{abstract}

\section{Introduction} \label{sec:intro}

A matrix $A \in \C^{n\times n}$ is called normal if $AA^* = A^* A$, where $A^*$ denotes the Hermitian transpose of $A$. Equivalently, the Hermitian and skew-Hermitian parts of $A$ commute:
\begin{equation} \label{eq:commuteHSH}
 HS = SH, \quad H:= (A+A^*)/2, \quad S:= (A-A^*) /2.
\end{equation}
It is a basic linear algebra fact that $A$ is normal if and only if it can be diagonalized by a unitary matrix $U$, that is, $U^* A U$ is diagonal.

Our recent work~\cite{hekressner2024randomized} on diagonalizing commuting matrices suggests a simple method for performing the diagonalization of a general normal matrix $A$: Draw two independent random numbers $\mu_H, \mu_S \in \R$ from the standard normal distribution $\mathcal N(0,1)$ and compute a unitary matrix $U$ that diagonalizes the \emph{Hermitian} matrix $\mu_H H + \mu_S \mathrm{i} S$. This results in the following algorithm:
\begin{algorithm}[H]
    \caption*{\textbf{Rand}omized \textbf{Diag}onalization of a normal matrix  ({\bf \randdiag})}
    \textbf{Input:} \text{Normal matrix $A \in \complex^{n \times n }$.}\\
     \textbf{Output:} \text{Unitary matrix $U$ such that $U^* A U$ is diagonal.} \\[-0.5cm]
    \begin{algorithmic}[1]
    \label{alg:RanDNorm}
        \STATE Compute $H = (A + A^*) /2, S = (A - A^*) /2$.
        \STATE Draw $\mu_H, \mu_S \in \mathcal N(0,1)$ independently.
        \STATE \label{line:evp} Compute unitary matrix $U$ such that $U^*(\mu_H H + \mu_S \mathrm{i} S)U$ is diagonal.
    \end{algorithmic}
    \end{algorithm}
    
As we will explain in Section~\ref{sec:analysis}, \randdiag{} succeeds with probability one in exact arithmetic and remains fairly robust in the presence of errors introduced due to, e.g., roundoff.
The main computational cost is in line~\ref{line:evp}, which directly benefits from decades of research and development on Hermitian eigenvalue solvers~\cite{Bai2000,Golub2013,Stewart2001}
and software, such as LAPACK~\cite{lapack99}. This makes \randdiag{} very simple to implement in scientific computing software environments. For example, these are the Matlab and Python implementations of \randdiag{}:

\begin{Verbatim}[frame=single]
H = (A+A')/2; S = (A-A')/2;
[U,~] = eig(randn*H+randn*1i*S);
\end{Verbatim}

\begin{Verbatim}[frame=single]
H = (A+A.conj().T)/2; S = (A-A.conj().T)/2
_,U = eigh(np.random.normal()*H + np.random.normal()*1j*S)
\end{Verbatim}

Compared to the Hermitian case, the case of a general normal matrix $A$ is much less developed. In particular, we are not aware of a single publicly available implementation of an eigenvalue solver tailored to (complex) normal matrices, despite the fact that the development of such algorithms is a classical topic in numerical analysis.
Already in 1959, Goldstine and Horwitz~\cite{GoldstineHorwitz59} proposed and analyzed a Jacobi-like method 
that annihilates the off-diagonal part of $A$ by alternatingly applying Givens rotations. Ruhe~\cite{Ruhe67}
established local quadratic convergence of this method when applying Givens rotations in cyclic order and developed later on, in~\cite{Ruhe87}, a modified version for finding the normal matrix that is nearest (in Frobenius norm) to a given matrix. Since then, research on Jacobi methods for normal matrices has fallen nearly dormant, with the notable exception of~\cite{MR4777800}.
However, two closely related problems have continued to attract interest: (1) Several algorithms have been developed for the more general task of simultaneously diagonalizing (nearly) commuting Hermitian matrices, including
Jacobi-like methods~\cite{Bunse-Gerstner93,JADE}, optimization-based methods~\cite{uwedge,ffdiag}, a ``do-one-then-do-the-other'' (DODO) approach~\cite{Sutton23}, 
as well as randomized methods; see~\cite{hekressner2024randomized} and the references therein.
By applying them to the commuting Hermitian matrices $H, \mathrm{i}S$, any of these algorithms can be used to diagonalize a normal matrix $A = H + S$. Very recently~\cite{Mataigne2024}, the DODO approach was adapted to block diagonalize a \emph{real} normal matrix with a \emph{real} orthogonal matrix.
(2) When $A$ is unitary, its Hessenberg form can be described by $\mathcal O(n)$ parameters, the so called Schur parameters, which is the basis of efficient
algorithms~\cite{UnitaryDC,UnitaryQR} for diagonalizing $A$. Unitary eigenvalue problems play an important role in mathematical physics; a recent application to 
modelling thermal conductivity  can be found in~\cite{LongPhysRevLett23}. The unitary variant of the QR algorithm described in~\cite{UnitaryQR} has been implemented in the {\tt eiscor} Fortran 90 package.\footnote{See \url{https://github.com/eiscor/eiscor}.}

\randdiag{} is not only very simple but it is also fast, as demonstrated by our numerical experiments in Section~\ref{sec:numexp}. Experience~\cite{Drmac2007} has shown that significant efforts are needed to make Jacobi-like methods, like the one by Goldstine and Horwitz, competitive. Even for the special case of a unitary matrix $A$, already the initial reduction to Hessenberg form required by the unitary QR algorithm can be more costly than \randdiag{}.
Note, however, that not all applications may require this reduction; see~\cite{robel21} for an example.

\section{Analysis of \randdiag} \label{sec:analysis}


We start our analysis of \randdiag{} by showing that it almost surely returns the correct result in the absence of error.
\begin{theorem}
\label{thm:rjd_exact_recovery}
Given a normal matrix $A$, the unitary  matrix $U$ returned by \randdiag{} diagonalizes $A$ with probability $1$.
\end{theorem}
\begin{proof}
The result follows from Theorem 2.2 in~\cite{hekressner2024randomized}, which is concerned with the more general situation of diagonalizing a family of commuting matrices. For the convenience of the reader, we reproduce and simplify the argument for the situation at hand.

Because the matrices $H$ and $S$ defined in~\eqref{eq:commuteHSH} commute, there exists a unitary matrix $U_0$ jointly diagonalizing $H$ and $\mathrm i S$:
\begin{equation} \label{eq:jointdiag}
 U_0^* H U_0 = \Lambda_H, \quad \mathrm U_0^* \mathrm{i} S U_0 = \Lambda_S,
\end{equation}%
with real diagonal matrices $\Lambda_H = \text{diag}( \lambda_1^{(H)},\ldots,
\lambda_n^{(H)})$, $\Lambda_S = \text{diag}( \lambda_1^{(S)},\ldots,
\lambda_n^{(S)})$. Clearly, this implies $
 U_0^* (\mu_H H + \mu_S \mathrm i S) U_0 = \mu_H  \Lambda_H + \mu_S \Lambda_S.$
On the other hand, the matrix $U$ computed by \randdiag{} also diagonalizes the same matrix: $U^*( \mu_H H + \mu_S \mathrm i S) = \widetilde \Lambda$.
The diagonal matrix $\widetilde \Lambda$ contains the same eigenvalues as $\mu_H  \Lambda_H + \mu_S \Lambda_S$, but they may appear in a different order on the diagonal.
Hence, by reordering the diagonal entries of $\widetilde \Lambda$, there is a permutation matrix $\Pi$ such that
\[
\Pi^* U^* (\mu_H H + \mu_S \mathrm i S) U \Pi = \Pi^* \widetilde \Lambda \Pi = \mu_H  \Lambda_H + \mu_S \Lambda_S.
\]
Defining $V:= U_0^* U \Pi$, the two relations above imply
\[
 V^* (\mu_H \Lambda_H + \mu_S \Lambda_S) V = \mu_H \Lambda_H + \mu_S \Lambda_S.
\]
In other words, the unitary matrix $V$ commutes with the diagonal matrix $\mu_H \Lambda_H + \mu_S \Lambda_S$, which implies that
$v_{ij} = 0$ for any $i,j \in \{1,\ldots,n\}$ such that
\begin{equation} \label{eq:differenteigenvalues}
\mu_H  \lambda_i^{(H)} + \mu_S \lambda_i^{(S)} \not= \mu_H  \lambda_j^{(H)} + \mu_S \lambda_j^{(S)}.
\end{equation}
Now consider any $i,j$ such that $\lambda_i^{(H)} \not= \lambda_j^{(H)}$ or $\lambda_i^{(S)} \not= \lambda_j^{(S)}$.
Because  $\mu_H, \mu_S$ are independent continuous random variables, the relation
$\mu_H  (\lambda_i^{(H)} - \lambda_j^{(H)}) + \mu_S (\lambda_i^{(S)} - \lambda_j^{(S)}) = 0$
is satisfied with probability zero; see, e.g., \cite{hekressner2024randomized}[Lemma 2.1]. Therefore, with probability $1$, the following hold: The relation~\eqref{eq:differenteigenvalues} and, thus, $v_{ij} = 0$ are satisfied for all such $i,j$. In turn, $V$ commutes with $\Lambda_H$ and with $\Lambda_S$. Multiplying the two relations~\eqref{eq:jointdiag} from the left with $V^*$ and from the right with $V$, we get
$U^* H U = \Lambda_H$ and $\mathrm U^* \mathrm{i} S U = \Lambda_S$.
In particular, $U$ diagonalizes $A = H+S$.\end{proof}

The proof of Theorem~\ref{thm:rjd_exact_recovery} reveals why randomness is needed in \randdiag{}.
Suppose that, instead of choosing $\mu_H,\mu_S$ randomly, we hard-coded a choice $\mu_H = \alpha$, $\mu_S = \beta$ for arbitrary fixed $\alpha, \beta \in \mathbb R$
in \randdiag{}. Then \randdiag{} would fail for the matrix
\[
 A = U_0 \begin{bmatrix}
      \beta + \mathrm{i} \alpha  & 0 \\
      0 & 0
     \end{bmatrix} U_0^*,
\]
with some unitary matrix $U_0 \not= I$. In this case, $\mu_H H + \mu_S \mathrm i S = 0$ and $U=I$ would be a viable (and likely) output of \randdiag{}. But, clearly, $U = I$ does not diagonalize $A$. 

When working in finite-precision arithmetic, the assumption of Theorem~\ref{thm:rjd_exact_recovery}, that $A$ is exactly normal, is not reasonable, unless $A$ has a particularly simple structure, like Hermitian or diagonal. Already representing the entries of a, say, unitary matrix as floating point numbers introduces an error (on the order of the unit round-off $u$) that destroys normality. The Hermitian eigenvalue solver utilized in line~\ref{line:evp} of \randdiag{} introduces additional error. If a backward stable method is used in line~\ref{line:evp}, one can conclude that \randdiag{} is executed for a slightly perturbed matrix $A+E$, with $\|E\|_F = \mathcal O( u \|A\|_F)$. However, it is unlikely that $E$ preserves normality; the best one can reasonably hope for in that situation is that the transformation returned by \randdiag{} diagonalizes $A$ up to an off-diagonal error proportional to $\|E\|_F$. In the following, we indeed establish such a robustness property by utilizing
results on randomized joint diagonalization from~\cite{hekressner2024randomized}. Note that $\offdiag(\cdot)$ refers to the off-diagonal part of a matrix, obtained by setting its diagonal entries to zero.
\begin{theorem}
\label{thm:robust-pair}
Consider $n\times n$ Hermitian matrices $A_1, A_2, \tilde A_1 = A_1 + E_1$, $\tilde A_2 = A_2 + E_2$ such that 
$A_1 A_2 = A_2 A_1$.
Let $\tilde U$ be a unitary matrix that diagonalizes $\mu_1 \tilde A_1 + \mu_2 \tilde A_2$ for independent $\mu_1,\mu_2 \sim \mathcal N(0,1)$.
Then, for any $R > 1$, the inequality
\[\big( \big\|\offdiag(\tilde{U}^T\tilde{A}_1\tilde{U})\big\|_F^2 + \big\|\offdiag(\tilde{U}^T\tilde{A}_2\tilde{U})\big\|_F^2 \big)^{1/2} \leq R \big( \|E_1\|_F^2 + \|E_2\|_F^2 \big)^{1/2}   \]
holds with probability at least $1-\frac{12}{\sqrt{\pi}}  \frac{n^{3.5}}{R-1}$.
\end{theorem}
\begin{proof}
The result follows directly from Theorem 3.6 in~\cite{hekressner2024randomized} for $d=2$ matrices, after a straightforward extension from the real symmetric to the complex Hermitian case. In particular, it can be easily verified that the invariant subspace perturbation bound from Lemma 3.2 in~\cite{hekressner2024randomized}, which plays a crucial role in the proof, continues to hold in the complex case. 
\end{proof}

\begin{corollary}[Robustness of \randdiag{} to error]\label{cor:robust_recovery}
Given a matrix $\tilde A = A + E \in \C^{n\times n}$, such that $A$ is normal, let $\tilde U$ denote the unitary matrix returned by \randdiag{} when applied to $\tilde A$. Then, for any $R > 1$, the inequality
    \[\|\offdiag(\tilde U^*\tilde A \tilde U)\|_F \leq R \|E\|_F \]
holds with probability at least $1 - \frac{12}{\sqrt{\pi}}  \frac{n^{3.5}}{R-1}$.
\end{corollary}
\begin{proof}
In analogy to the decomposition $A = H+S$ from~\eqref{eq:commuteHSH}, we decompose $\tilde A = \tilde H + \tilde S$ with
\[\tilde H = H +  E_H, \quad \tilde S = S + E_S, \quad E_H = (E+E^*)/2, \quad E_S = (E-E^*)/2. \]
Because Hermitian and skew-Hermitian matrices are orthogonal to each other (in the Frobenius inner product) and taking off-diagonal parts does not affect this property, it follows that
\begin{align*}
        \|\offdiag(\tilde U^*\tilde A \tilde U)\|_F^2 =&\|\offdiag(\tilde U^*\tilde H \tilde U)  + \offdiag(\tilde U^*\tilde S \tilde U) \|^2_F\\
        =& \|\offdiag(\tilde U^*\tilde H \tilde U) \|_F^2 +\|\offdiag(\tilde U^* \mathrm{i} \tilde S\tilde U)\|_F^2,
\end{align*}
as well as $\|E\|_F^2  = \|E_H\|_F^2 + \|\mathrm{i} E_S\|_F^2$.  The result of the corollary follows from Theorem~\ref{thm:robust-pair} by setting $A_1 = H$, $A_2 = \mathrm{i}S$, $E_1 = E_H$, and $E_2 = \mathrm{i} E_S$.
%
\end{proof}

The result of Corollary~\ref{cor:robust_recovery} guarantees, with high probability, an off-diagonal error proportional to the input error, if one allows for a magnification of the error by a factor that grows polynomially with $n$. Note that Corollary~\ref{cor:robust_recovery} also implies a backward error result because $\tilde U$ diagonalizes the normal matrix $\tilde A - \tilde U\offdiag(\tilde U^*\tilde A\tilde U)\tilde U^*$. In turn, this allows one to apply perturbation results, such as the Hoffmann-Wielandt theorem~\cite[Thm. IV.3.1]{Stewart1990}, to draw conclusions on, e.g., the quality of the approximate eigenvalues obtained from the diagonal of $\tilde U^*\tilde A \tilde U$.

Note that, both, Theorem~\ref{thm:robust-pair} and Corollay~\ref{cor:robust_recovery} make the (simplifying) assumption that the random variables involved are not affected by perturbations.

\section{Numerical experiments} \label{sec:numexp}

In this section, we present a few numerical experiments to indicate the efficiency of \randdiag{}. With no (efficient) implementation of the Jacobi-like methods mentioned in Section~\ref{sec:intro} being available, the main competitor to \randdiag{} appears to be the (non-symmetric) QR algorithm~\cite{Golub2013} to compute the unitary matrix that transforms $A$ into Schur form. Thanks to the backward stability of the QR algorithm, such an approach enjoys a strong robustness property in the sense of Corollary~\ref{cor:robust_recovery}.

All experiments were carried out on a Dell XPS 13 2-In-1 with an Intel Core i7-1165G7 CPU and 16GB of RAM. We have implemented\footnote{Our implementation as well as all files needed to reproduce the experiments can be found at \url{https://github.com/haoze12345/Diagonalizing-Normal-Matrices}.} \randdiag{} in Python $3.8$, calling LAPACK routines via Scipy for computing spectral and Schur decompositions. All experiments have been carried out in double precision and all execution times/errors are averaged over $100$ runs with fixed input matrices.    For the computed matrix $\tilde U$, we report the off-diagonal error $\|\offdiag(\tilde U^*A \tilde U)\|_F$.

\subsection{Synthetic data}

We have tested \randdiag{} on random complex unitary matrices obtained from applying the QR factorization to $n\times n$ complex Gaussian random matrices. 
The obtained results are summarized in Table~\ref{table:synthetic1_500} to Table~\ref{table:synthetic1_1500}.%


\begin{table}[!hbt!]
\begin{center}
\caption{Execution time and off-diagonal error for \randdiag{} vs. Schur decomposition applied to a random unitary matrix with $n=500$.}

\label{table:synthetic1_500}
\small
\begin{tabular}{|c|c|S[table-format=2.3]|S[table-format=2.3]|S[table-format=2.3]|S[table-format=2.3]|}
\hline
&{time} & {Error mean} & {Error std} & {Error min} & {Error max} \\
\hline
RandDiag & {0.16s}& $\num{1.42e-10}$ & $\num{2.56e-10}$ & $\num{6.13e-12}$ & $\num{1.58e-09}$ \\
Schur & {0.64s} &$\num{2.58e-13}$ & $\num{ 5.05e-29}$& $\num{2.58e-13}$ & $\num{2.58e-13}$\\
\hline
\end{tabular}
\end{center}
\end{table}

\begin{table}[!hbt!]
\begin{center}
\caption{Execution time and off-diagonal error for \randdiag{} vs. Schur decomposition applied to a random unitary matrix  with $n=1000$.}

\label{table:synthetic1_1000}
\small
\begin{tabular}{|c|c|S[table-format=2.3]|S[table-format=2.3]|S[table-format=2.3]|S[table-format=2.3]|}
\hline
&{time} & {Error mean} & {Error std} & {Error min} & {Error max} \\
\hline
RandDiag &{0.58s} & $\num{7.88e-10}$ & $\num{3.60e-09}$ & $\num{ 2.54e-11}$ & $\num{3.54e-08}$ \\
Schur & {2.47s} &$\num{4.84e-13}$ & $\num{0}$& $\num{4.84e-13}$ & $\num{4.84e-13}$\\
\hline
\end{tabular}
\end{center}
\end{table}

\begin{table}[!hbt!]
\begin{center}
\caption{Execution time and off-diagonal error for \randdiag{} vs. Schur decomposition applied to a random unitary  matrix with $n=1500$.}

\label{table:synthetic1_1500}
\small
\begin{tabular}{|c|c|S[table-format=2.3]|S[table-format=2.3]|S[table-format=2.3]|S[table-format=2.3]|}
\hline
&{time} & {Error mean} & {Error std} & {Error min} & {Error max} \\
\hline
RandDiag & {1.45s} & $\num{1.24e-09}$ & $\num{3.92e-09}$ & $\num{9.64e-11}$ & $\num{3.62e-08}$ \\
Schur & {6.50s} &$\num{6.72e-13}$ & $\num{ 0}$& $\num{6.72e-13}$ & $\num{6.72e-13}$\\
\hline
\end{tabular}
\end{center}
\end{table}
It turns out that \randdiag{} is up to four times faster than the standard Schur decomposition, at the expense of a few digits (on average three, at most five)
of accuracy in the off-diagonal error. Note that this does not necessarily translate into reduced eigenvalue accuracy. Indeed, for the matrices from Table~\ref{table:synthetic1_500} to Table~\ref{table:synthetic1_1500}, extracting the eigenvalues from the diagonal of $U^* A U$ for the matrix $U$ returned by \randdiag{} results in accuracy comparable to the Schur decomposition. In analogy to results for Hermitian matrices (see~\cite{Nakatsukasa2017} and the references therein), the eigenvalue error appears to depend quadratically on the off-diagonal error for well separated eigenvalues.

To further analyze the accuracy of eigenvalues returned by \randdiag{},  we compute relative errors of the output eigenvalues for randomly generated normal matrices $A \in \complex^{n \times n}$ of the form $U D U ^*$. The unitary matrix $U$ is obtained as in the last experiment above, and $D$ is diagonal with diagonal entries sampled from standard complex Gaussian. The relative errors are computed as follows. Let $\diag(\cdot)$ denote the vector of diagonal entries of a matrix and $\tilde U$ be the output of \randdiag{} applied to $A$. Define $d_1 = \diag(D) \in \mathbb{C}^n$ as the vector of original eigenvalues and $d_2 = \diag(\tilde U ^* A \tilde U) \in \mathbb{C}^n$ as the vector of the output eigenvalues. The relative error is given by: 
$$\|d_1 - P d_2\|_2 / \|d_1\|_2$$ where $P$ is the the permutation matrix minimizing $\|d_1 - P d_2\|_2$. This optimal permutation is  computed by solving a simple linear sum assignment problem, for which we leverage the implementation in {\tt Scipy}~\footnote{
\url{https://docs.scipy.org/doc/scipy/reference/generated/scipy.optimize.linear_sum_assignment.html}
}. The results, summarized in Table \ref{table:eigenvalue_error500} to Table \ref{table:eigenvalue_error1500}, show that \randdiag{} tends to be slightly more accurate than the Schur decomposition for this data, whereas the Schur decomposition yields more consistent results.
\begin{table}[!hbt!]
\begin{center}
\caption{Eigenvalue relative errors for \randdiag{} vs. Schur decomposition for a random normal matrix with $n=500$.}

\label{table:eigenvalue_error500}
\small
\begin{tabular}{|c|S[table-format=2.3]|S[table-format=2.3]|S[table-format=2.3]|S[table-format=2.3]|}
\hline
&{error mean} & {error std} & {error min} & {error max} \\
\hline
RandDiag & $\num{1.12e-15}$ & $\num{3.30e-17}$ & $\num{1.06e-15}$ & $\num{1.22e-15}$ \\
Schur & $\num{4.74e-15}$ & $\num{ 7.89e-31}$& $\num{4.74e-15}$ & $\num{4.74e-15}$\\
\hline
\end{tabular}
\end{center}
\end{table}

\begin{table}[!hbt!]
\begin{center}
\caption{Eigenvalue relative errors for \randdiag{} vs. Schur decomposition for a random normal matrix with $n=1000$.}

\label{table:eigenvalue_error1000}
\small
\begin{tabular}{|c|S[table-format=2.3]|S[table-format=2.3]|S[table-format=2.3]|S[table-format=2.3]|}
\hline
&{error mean} & {error std} & {error min} & { error max} \\
\hline
RandDiag & $\num{1.57e-15}$ & $\num{3.96e-17}$ & $\num{1.47e-15}$ & $\num{1.67e-15}$ \\
Schur & $\num{5.91e-15}$ & $\num{ 3.16e-30}$& $\num{5.91e-15}$ & $\num{5.91e-15}$\\
\hline
\end{tabular}
\end{center}
\end{table}

\begin{table}[!hbt!]
\begin{center}
\caption{Eigenvalue relative errors for \randdiag{} vs. Schur decomposition for a random normal matrix with $n=1500$.}

\label{table:eigenvalue_error1500}
\small
\begin{tabular}{|c|S[table-format=2.3]|S[table-format=2.3]|S[table-format=2.3]|S[table-format=2.3]|}
\hline
&{error mean} & {error std} & {error min} & {error max} \\
\hline
RandDiag & $\num{1.49e-15}$ & $\num{2.56e-17}$ & $\num{1.41e-15}$ & $\num{1.55e-15}$ \\
Schur & $\num{6.61e-15}$ & $\num{ 1.58e-30}$& $\num{6.61e-15}$ & $\num{6.61e-15}$\\
\hline
\end{tabular}
\end{center}
\end{table}

For unitary matrices, using a structure-preserving QR algorithm~\cite{UnitaryQR} reduces the time needed to compute the Schur decomposition. However, it is unlikely that even a well-tuned implementation will significantly outperform \randdiag{}, because of the required initial reduction to Hessenberg form. For $n = 1000$, Hessenberg reduction requires $0.44s$ in Python. In Matlab, we found that Hessenberg reduction requires $0.31$s but RandDiag only requires $0.21$s. 

Next, we examine the growth of error with respect to matrix size for \randdiag{}. Figure \ref{fig:error_growth_rate} illustrates the average off-diagonal errors over $100$ runs of randomly generated unitary matrices with sizes ranging from $8$ to $4096$, in powers of 2. Note that the plot uses a log-log scale. Empirically, the  growth rate is closer to $\mathcal{O}(n^2)$, indicating that the $\mathcal{O}(n^{3.5})$ bound in Corollary \ref{cor:robust_recovery} is too pessimistic. Tightening this $\mathcal{O}(n^{3.5})$ bound remains an open question for future research.
\begin{figure}[h]
    \centering
    \includegraphics[width=\linewidth]{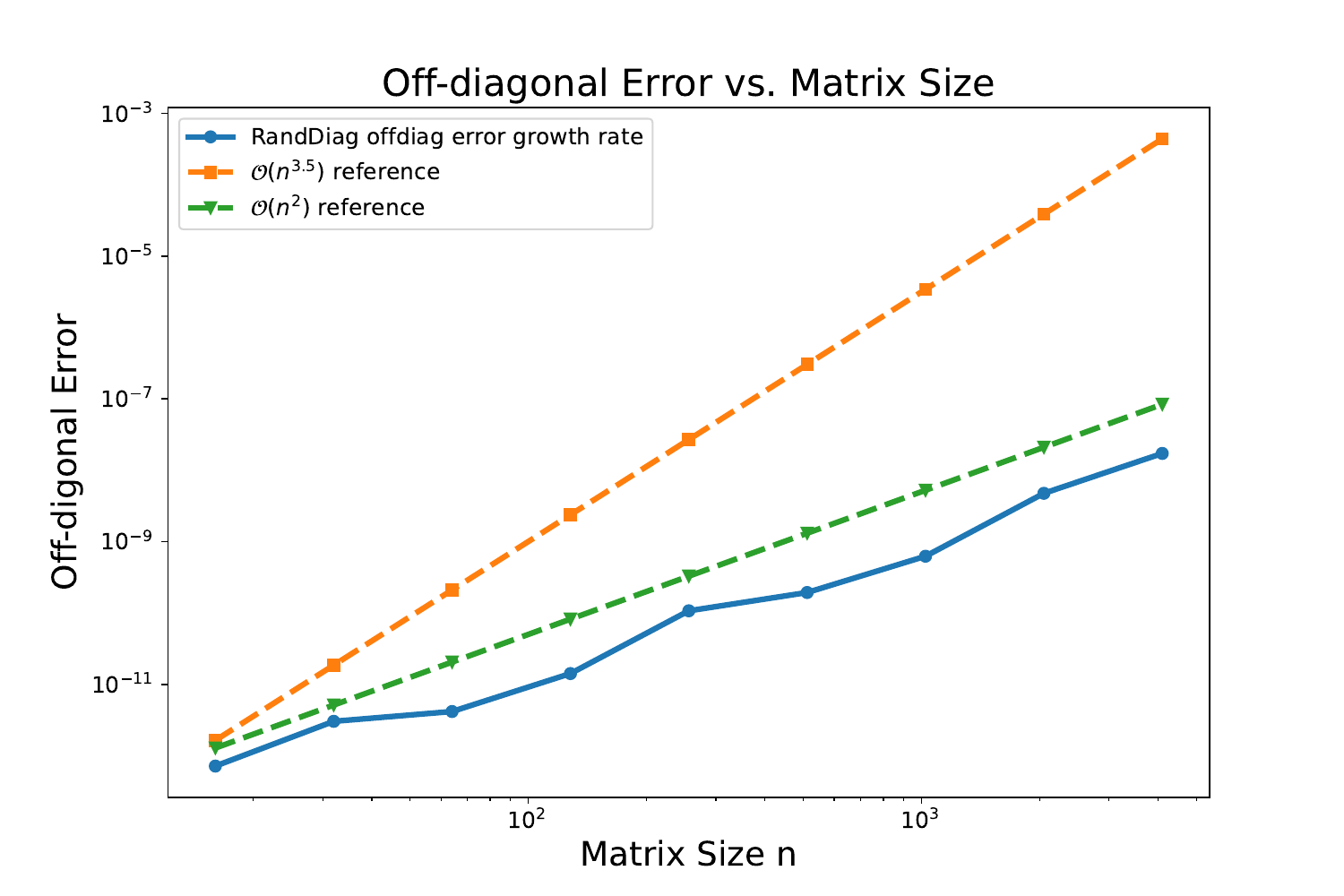}
    \caption{Error growth rate of RandDiag vs. matrix size on log-log scale}
    \label{fig:error_growth_rate}
\end{figure}

\subsection{Thermal conductivity estimation}

A thermal conductivity model from~\cite{LongPhysRevLett23} features unitary matrices of the form
\[U_F = U_{\text{int}}U_0, \quad U_0 = \bigotimes_{j=1}^{L}d_j, \quad U_{\text{int}} = \prod_{j=1}^{L-1} I_{2^{{\pi(j)}-1}} \otimes u_{\pi(j),\pi(j)+1} \otimes I_{2^{L-{\pi(j)}-1}}\]
where $\otimes$ denotes the usual Kronecker product, $L$ corresponds to the number of states, each $d_j$ is a (uniformly distributed) random $2 \times 2$ unitary matrix, and $u_{j,j+1} = \exp[\mathrm{i} M_{j,j+1}]$ where $M_{j,j+1}$ is a random matrix drawn from the $4 \times 4$ Gaussian unitary ensemble (GUE) normalized such that $\mathbb{E}[\text{trace}(M^2)] = 2$. 
Note that $U_{\text{int}}$ describes the nearest neighbor interaction between states. The result of \randdiag{} applied to $U_F$ for $L = 11$ (that is, $n = 2^{11}=2048$) is shown in Table \ref{table:real}. This time, we observe a speedup by nearly a factor of $6$, at the expense of more than $4$ digits of accuracy in the off-diagonal error. The level of accuracy attained by \randdiag{} usually suffices for this kind of applications.



\begin{table}[!hbt!]
\begin{center}
\caption{Execution time and off-diagonal error for \randdiag{} vs. Schur decomposition applied to thermal conductivity model}

\label{table:real}
\small
\begin{tabular}{|c|S[table-format=2.3]c|S[table-format=2.3]c|S[table-format=2.3]c|}
\hline
 & \multicolumn{1}{c}{time} & \multicolumn{1}{c|}{error} \\ 
\hline
RandDiag & {\space3.49s} & $\num{1.26e-09}$\\

Schur & {20.73s} & $\ \num{9.24e-13}$\\
\hline
\end{tabular}
\end{center}
\end{table}

\begin{paragraph}{Acknowledgments.} The authors thank Vjeran Hari, Zagreb, for suggesting the topic of this work.
\end{paragraph}

\section*{Declarations}
\subsection*{Competing interests}
The authors have no competing interests to declare that are relevant to the content of this article.
\subsection*{Data availability}
All the data is available at \url{https://github.com/haoze12345/Diagonalizing-Normal-Matrices}.

\bibliography{normal_bib}
\bibliographystyle{abbrv}
\end{document}